\documentclass[a4paper, twoside,12pt]{article}
\usepackage{fancyhdr}
\usepackage{fnpos}
 \usepackage[english]{babel}        
\usepackage[T1]{fontenc}          
\usepackage{graphicx}             
\usepackage{makeidx}
\usepackage{fancybox}
\usepackage{framed}
\usepackage{fancyhdr}
 \usepackage{pstricks,pst-plot,pstricks-add}
\usepackage[margin=1in]{geometry}
\usepackage{graphicx}
\usepackage{titlesec}
\usepackage{amsmath}
\usepackage{amsfonts}   
\usepackage{amssymb}    
\usepackage{amsthm}
\usepackage{dsfont}
\usepackage{mathtools}
\usepackage{verbatim}   
\usepackage{color}
\usepackage{listings}
\usepackage{graphicx}
\usepackage{amsmath}
\usepackage{amsmath,amssymb,color,inputenc,euscript,graphicx,psfrag}

\newtheorem{thm}{Theorem}[section]

\newtheorem{lem}[thm]{Lemma}
\newtheorem{prop}[thm]{Proposition}

\newtheorem{rem}[thm]{Remark}
\numberwithin{equation}{section}

\renewcommand{\thefootnote}

\newcommand\co{\operatorname{co}}

\renewcommand\Re{\operatorname{Re}}

 {\markboth{\sl B.  Amri  and M. Gaidi }{\sl Dunkl wave equation  }

\author {B\'echir Amri and  Mohamed Gaidi}

\title{  $L^p-L^q$ estimates for the solution  of the  Dunkl wave equation  }

\date{}
 
\begin{document}
 \maketitle
\begin{center}
   Universit\'{e} Tunis El Manar, Facult\'{e} des sciences de Tunis,\\ Laboratoire d'Analyse Math\'{e}matique
       et Applications,\\ LR11ES11, 2092 El Manar I, Tunisie.\\
     \textbf{ e-mail:} bechir.amri@ipeit.rnu.tn,  gaidi.math@gmail.com
\end{center}
  \begin{abstract}
In this paper, our main aim is to derive  $L^p-L^q$
 estimates of   the   solution $u_k(x,t)$ ( t fixed) of the Cauchy problem  for  the homogeneous linear wave equation  associated
to  the Dunkl Laplacian $\Delta_k$,  $$\Delta_ku_k(x,t)= \partial_t^2u_k (x,t),\quad
\partial_tu_k(x,0)= f(x),\quad u_k(x,0)= g(x).$$
We extend to Dunkl setting the  estimates  given   by Srichartz in \cite{Sti}   for the ordinary wave equation .
 \end{abstract}
  \section{ Introduction and Background}
   In his seminal paper \cite{D1}, Dunkl constructed  a family of differential-difference operators associated to
a finite reflection group on a Euclidean space, which are known by   Dunkl operators in the literature. Associated with the eigenfunctions of the
Dunkl operators and a  non negative multiplicity function,   Dunkl \cite{D2} introduced
  an integral transform on $\mathbb{R}^n$  called   Dunkl transform   which generalises the classical
Fourier transform. By means of this  transform several important results  of the classical Fourier analysis have been generalized to Dunkl analysis which
  can offer new perspectives on familiar topics from Harmonic Analysis and Partial Differential Equations.
 \par In this paper, we are interested by the $L^p-L^q$  estimates of the solutions of wave equations associated to Dunkl Laplace operator, in particular we generalize the estimates for ordinary wave equation due to Stchartz \cite{Sti}.
The techniques used involve  interpolation of  an appropriate  analytic families of operators  in a way similar
to that used in the classical  case. To begin we  first   provide  some preliminary definitions  and background materials for the Dunkl analysis . References are \cite{D1,D2,J1,R3,R4,Xu}
\par Let $G\!\subset\!\text{O}(\mathbb{R}^n)$
be a finite reflection group associated to a reduced root system $R$
and $k:R\rightarrow[0,+\infty)$  be a $G$--invariant function
(called multiplicity function).
Let $R^+$ be a positive root subsystem. The Dunkl operators  $D_\xi^k$ on $\mathbb{R}^n$ are
the following $k$--de\-for\-ma\-tions of directional derivatives $\partial_\xi$
by difference operators\,:
\begin{equation}\label{Dxi}
D_\xi^k f(x)=\partial_\xi f(x)
+\sum_{\,\upsilon\in R^+}\!k(\upsilon)\,\langle\upsilon,\xi\rangle\,
\frac{f(x)-f(\sigma_\upsilon.\,x)}{\langle\upsilon,\,x\rangle}\,,
\end{equation}
where here
 $\sigma_\upsilon$ is the reflection
with respect to the hyperplane orthogonal to $ \upsilon$ and $\langle .,\, .\rangle$ is the usual Euclidean
inner product, we denote  $| \,.\,  |$ its induced norm. If $(e_j)_{j}$ is the canonical basis of $\mathbb{R}^n$ we simply write
 $D^k_{j}$ instead of  $D^k_{e_j}$. In analogy to the ordinary  Laplacian we define  the Dunkl Laplacian operator by
$$\Delta_k=\sum_{j=1}^n(D^k_{j})^2.$$
The Dunkl operators are antisymmetric
with respect to the measure $w_k(x)\,dx$
with density
$$
w_k(x)=\,\prod_{\,\upsilon\in R^+}|\,\langle\upsilon,x\rangle\,|^{\,2\,k(\upsilon)}\,.
$$

The operators $\partial_\xi$ and $D_\xi^k$
are intertwined by a Laplace--type operator
\begin{eqnarray*}\label{vk}
V_k\hspace{-.25mm}f(x)\,
=\int_{\mathbb{R}^n}\hspace{-1mm}f(y)\,d\nu_x(y)
\end{eqnarray*}
associated to a family of compactly supported probability measures
\,$\{\,\nu_x\,|\,x\!\in\!\mathbb{R}^n\hspace{.25mm}\}$\,.
Specifically, \,$\mu_x$ is supported in the the convex hull $\co(G.x)\,.$
\par For every $y\!\in\!\mathbb{C}^n$\!,
the simultaneous eigenfunction problem
\begin{equation*}
D_\xi^k f=\langle y,\xi\rangle\,f,
\qquad\forall\;\xi\!\in\!\mathbb{R}^n,
\end{equation*}
has a unique solution $f(x)\!=\!E_k(x,y)$
such that $E_k(0,y)\!=\!1$, called the Dunkl kernel and is given by
\begin{equation*}\label{EV}
E_k(x,y)\,
=\,V_k(e^{\,\langle.,\,y\,\rangle})(x)\,
=\int_{\mathbb{R}^n}\hspace{-1mm}e^{\,\langle z,y\rangle}\,d\nu_x(z)
\qquad\forall\;x\!\in\!\mathbb{R}^n.
\end{equation*}
 When $ k =0$ the Dunkl kernel $E_k(x,y)$ reduces to the exponential $e^{\langle x,y\rangle}$.
\par The Dunkl transform  is defined on $L^1(\mathbb{R}^n\!,w_k(x)dx)$ by
$$
\mathcal{F}_kf(\xi)=c_k^{-1}
\int_{\mathbb{R}^n}\!f(x)\,E_k(x,-i\,\xi)\,w_k(x)\,dx\,,
$$
where
$$
c_k\,= \int_{\mathbb{R}^n}\!e^{-\frac{|x|^2}2}\,w(x)\,dx\, .
$$
If $k=0$  then Dunkl transform  coincides with the usual Fourier transform.
In the sequel, we denote by  $\|\,.\,\|_{p,k}$    the  norm  of $L^p(\mathbb{R}^n,w_k(x)dx)$,  $1\leq p<\infty$.
Below  we summarize some of the    useful
  properties of the Dunkl transform.
\begin{itemize}
 \item[(i)] For $f\in \mathcal{S}(\mathbb{R}^n)$  (the Schwartz space) we have   $\mathcal{F}_k(D_\xi^k f )(x)=i\langle\xi,x\rangle\mathcal{F}_k(f)(x)$.
In particular $$\mathcal{F}_k(\Delta_k f )(x)=-|x|^2\mathcal{F}_k(f)(x).$$
\item[(ii)]
The Dunkl transform is a topological automorphism
of the Schwartz space $\mathcal{S}(\mathbb{R}^n)$.
\item[(iii)]
(\textit{Plancherel Theorem\/})
The Dunkl transform extends to
an isometric automorphism of $L^2(\mathbb{R}^n\!,w_k(x)dx)$.
\item[(iv)]
(\textit{Inversion formula\/})
For every $f\!\in\!\mathcal{S}(\mathbb{R}^n)$,
and more generally for every $f\!\in\!L^1(\mathbb{R}^n\!,w_k(x)dx)$
such that $\mathcal{F}_kf\!\in\!L^1(\mathbb{R}^n\!,w_k(x)dx)$,
we have
$$
f(x)=\mathcal{F}_k^2\!f(-x)\qquad\forall\;x\!\in\!\mathbb{R}^n.
$$
\item[( vi)] The Hausdorff-Young inequality: if $1<p\leq 2$, then
\begin{equation}\label{Hausd}
  \|\mathcal{F}_k(f)\|_{1,p'}\leq \|f\|_{1,p}
\end{equation}
where $p'$ is the conjugate exponent of $p$.
\item[( v )]
If $f$ is a radial function in $L^1(\mathbb{R}^n\!,w_k(x)dx)$ such that  $f(x)=\widetilde{f}(|x|)$, then
$\mathcal{F}_k(f)$ is also radial and
\begin{equation}\label{rad}
    \mathcal{F}_k(f)(x)= \frac{1}{ |x|^{\gamma_k+n/2-1}}\int_0^\infty\widetilde{f}(s)  J_{\gamma_k+n/2-1}(s|x|)\;s^{ \gamma_k+n/2}ds; \quad x\in\mathbb{R}^n.
\end{equation}
where
\begin{equation}\label{gk}
 \gamma_k=\sum_{\upsilon\in R^+} k(\upsilon)
\end{equation}
  and  $J_\nu$ is the Bessel function,
$$J_\nu(z)=\sum_{n=0}^\infty\frac{(-1)^n(z/2)^{2n+\nu}}{n!\Gamma(n+\nu+1)}.$$
\end{itemize}
\par Let $x\in \mathbb{R}^n$, the
Dunkl translation operator $f\rightarrow\tau_x(f)$ is defined on $L^2_k(\mathbb{R}^n,w_k(x)dx)$ by
\begin{eqnarray*}\label{dutr}
\mathcal{F}_k(\tau_x(f))(y)= \mathcal{F}_kf(y)\,E_k(x,iy), \quad
y\in\mathbb{R}^n.
\end{eqnarray*}
In particular, when $f\in S(\mathbb{R}^n)$ we have
\begin{equation*}
  \tau_x(f)(y)=\int_{\mathbb{R}^n}\mathcal{F}_k(f)(\xi)E_k(x,i\xi)E_k(y,-i\xi)w_k(\xi)d\xi.
\end{equation*}
and
\begin{equation}\label{trs}
  \|\tau_x(f)\|_{\infty,k}\leq \|\mathcal{F}_k(f)\|_{1,k}.
\end{equation}
  In the case when $f(x)=\widetilde{f}(|x|)$ is a  continuous radial function  that belongs to   $ L^2(\mathbb{R}^n\!,w_k(x)dx)$,  the Dunkl translation is represented by the following integral \cite{Dai},
\begin{eqnarray}\label{trad}
\tau_x(f)(y)=
\int_{\mathbb{R}^{n}}\widetilde{f}\left( \sqrt{|y|^2+|x|^2+2\langle y,\eta\rangle}\;\right)d\nu_x(\eta),
 \end{eqnarray}
This formula shows that the Dunkl translation operators can   be extended to all radial functions $f$ in $L^p (\mathbb{R}^n,w_k(x)dx)$, $1\leq p\leq \infty$  and the following holds
\begin{equation}\label{trp}
  \|\tau_x(f)\|_{p,k}\leq \|f\|_{p,k }\,.
  \end{equation}
 \par We define  the Dunkl convolution product   for suitable functions $f$ and $g$ by
$$f*_kg(x)=\int_{\mathbb{R}^N} \tau_x(f)(-y)g(y)d\mu_k(y),\quad x\in\mathbb{R}^n.$$
We note that it is commutative and satisfies the following property:
\begin{eqnarray}\label{conv}
 \mathcal{F}_k(f*_kg)=\mathcal{F}_k(f)\mathcal{F}_k(g).
\end{eqnarray}
Moreover, the operator $ f \rightarrow f*_kg $ is bounded on $L^p (\mathbb{R}^n,w_k(x)dx)$
provide $g$ is a bounded radial function in $L^1(\mathbb{R}^n,w_k(x)dx)$. In particular we have the   the following Young's inequality:
\begin{equation}\label{Y}
 \|f*_kg\|_{p,k}\leq \|g\|_{1,k}\|f\|_{p,k}\;.
\end{equation}
 \par We now come to the main subject. We consider the following Cauchy problem for the Dunkl wave equation
$$\Delta_k u_k(x,t)=\partial_t^2 u_k(x,t),\quad\partial_t u_k(x,0)=f(x),\quad u_k(x,0)=g(x); \quad (x,t)\in \mathbb{R}^n\times \mathbb{R}$$
 where the functions $f$ and $g$ belong to $S(\mathbb{R}^n)$.
The solution  is  given in term  of Dunkl transform by
\begin{equation}\label{wave}
   u_k(x,t) =\mathcal{F}_k^{-1}\left(\frac{\sin t |\xi| }{ |\xi| } \mathcal{F}_kf(\xi)+ \cos( t|\xi|) \mathcal{F}_k(g)(\xi)\right)(x).
\end{equation}
The study of the Dunkl  wave equation  was initiated by Ben Sa\"{\i}d and  {\O}rsted   \cite{said1} where  they  computed  the solution $u_k$ and established  validity of Huygens' Principles.  Also in   this context  Mejjaoli \cite{Majj1,Majj2}
  studied the  mixed-norm Strichartz type estimates for  $u_k$.
 The main contribution of our work is  the following theorem.
\begin{thm}\label{thm2}
For $t\neq0$ there exists $C(t)>0$  such that
\begin{equation}\label{eqthm1}
   \|u(.,t)\| _{q,k}\leq C(t)\left( \|f\|_{p,k}+ \left\| \sum_{j=1}^n|D_j^kg|\right\|_{p,k}\right),
\end{equation}
provide that
\begin{equation}\label{q1}
2\; \frac{n+2\gamma_k+1}{n+2\gamma_k+3}\leq p\leq2;\quad \frac{n+2\gamma_k}{q}=\frac{2\gamma_k+n-1}{2}-\frac{1}{p'}
\end{equation}
and
\begin{equation}\label{q2}
 2\frac{n+2\gamma_k }{n+2\gamma_k+2}\leq p\leq2\frac{n+2\gamma_k+1}{n+2\gamma_k+3};\quad \frac{1}{q}=\frac{n+2\gamma_k-1}{2}-\frac{n+2\gamma_k}{p'},
\end{equation}
where $\gamma_k$ is given by (\ref{gk}) and $p'$ is the conjugate exponent of $p$.
\end{thm}
The Proof is  based on complex interpolation  method much like the
proof  given in \cite{Sti}.  For the reader's convenience we recall   the Stein's  Interpolation Theorem \cite{Stein2}.
\par Let $(X,M, \mu)$ and $(Y, N, \nu)$ be $\sigma$-finite measure spaces and
$$S=\{z\in\mathbb{C};\;a\leq Re(z)\leq b\}, \quad a<b.$$
We suppose that we are given a linear operator $T_z$, for each $z \in S$, on the
space of simple functions in $L^1(M,\mu)$ into the space of measurable functions
on $N$. If $f$ is a simple function in $L^1(M,\mu)$ and and $g$ a simple function
in $L^1(N,\nu)$, we assume furthermore that $gT_z(f)\in L^1(N,\nu)$.
The family of operators $\{T_z\}$ is called admissible if
 the mapping
$$F:\; z\rightarrow \int_Ng\,T_z(f)d\nu$$
is holomorphic in the interior of $S$
 and continuous on S, and  there exists a
constant $c< \pi(b-a)$ such that
\begin{equation}\label{log}
 \sup_{z\in S}e^{-c|Im(z)|}\log|F(z)|<\infty.
\end{equation}
\begin{thm}[Stein,\cite{Stein2}]\label{sstt} Let $1 \leq  p_0, p_1, q_0, q_1 \leq  \infty $  and $\{T_z\}$, $z\in S$ , be  an admissible family of linear operators such
$$\|T_{a+iy}(f)\|_{q_0}\leq M_0(y)\|f\|_{p_0}\quad\text{and}\quad \|T_{b+iy}(f)\|_{q_1}\leq M_1(y)\|f\|_{p_1}$$
for each real number y  and  each simple function $ f\in L^1(M,\mu)$. If, in addition, the constants $M_j (y)$, $j=1,2$,
satisfy
$$e^{-c|y|}\log(M_j(y))<\infty$$
for some $c < \pi(b-a)$,  then for all  $t \in [0, 1]$ there exists a constant $M_t$ such that
 $$\|T_{\theta_t}(f)\|_{q_t}\leq M_t \|f\|_{p_t}$$
for all simple functions $f$ provided
$$\theta_t=(1-t)a+tb, \quad \frac{1}{p_t}=\frac{(1-t)}{p_0} +\frac{t}{p_1} \quad\text{and}\quad \frac{1}{q_t}=\frac{(1-t)}{q_0} +\frac{t}{q_1}$$
\end{thm}
\section{ Interpolation of analytic family of operators }
The idea now is to   consider the following family of operators $f\rightarrow S_ z(f)$ given on $L^2(\mathbb{R}^n,w_k(x)dx)$ by
 $$S_ z(f)(x)= \mathcal{F}_k^{-1}\Big(|\xi|^{\gamma_k+n/2-z} J_{\gamma_k+n/2-z}(|\xi|)\mathcal{F}_k(f)(\xi)\Big)(x)$$
where $z$ can be taken to be complex. Hence
in view of (\ref{wave}) and by the fact that
$$J_{1/2}(t)=\left(\frac{2}{\pi}\right)^{1/2}\;\frac{\sin t}{\sqrt{t}}\quad\text{and}\quad J_{-1/2}(t)=\left(\frac{2}{\pi}\right)^{1/2}\;\frac{\cos t}{\sqrt{t}}$$
one can write
\begin{equation}\label{sol}
  u_k(x,t)=\left(\frac{\pi}{2}\right)^{1/2} t\delta(t)S_{ \gamma_k+(n+1)/2}\delta(t^{-1})f(x)+ \left(\frac{\pi}{2}\right)^{1/2} \delta(t)S_{  \gamma_k+(n-1)/2}\delta(t^{-1})g(x)
\end{equation}
 where $\delta(t)$ is the dilation operator
$ \delta(t)f(x)=f(tx).$
\par The operator  $S_z $ turns out to be the analytic continuation  in the parameters $z$ of the  convolution operator
 $$T_z(f) =\Phi_z*_kf,\quad  0\leq \Re(z)<1$$
where
\begin{equation*}
 \Phi_z(x)= \left\{
    \begin{array}{ll}
      \frac{2 ^{z } }{\Gamma(1-z)} (1-|x|^2) ^{-z}, & \hbox{ if $|x|\leq 1$},\\
0, & \hbox{ if $|x| > 1$.}
    \end{array}
  \right.
\end{equation*}
This is a consequence of the following proposition.
\begin{prop}\label{prop}
The Dunkl  transform of $\Phi_z,\;\; 0\leq \Re(z)<1,\;\;$ is given by
\begin{equation}\label{Bes}
  \mathcal{F}_k(\Phi_z)(\xi)=  |\xi|^{z-\gamma_k-n/2} J_{\gamma_k+n/2-z}( |\xi|).
\end{equation}
\end{prop}
\begin{proof}
Since $\Phi_z$ is a radial function that belongs to $L^1(\mathbb{R}^n,w_k(x)dx)$ then from (\ref{rad}) we have
\begin{equation*}
    \mathcal{F}_k(\Phi_z)(\xi)=\frac{2 ^{z } }{\Gamma(1-z)  |\xi|^{\gamma_k+n/2-1}}\int_0^1(1-s^2)^{-z}  J_{\gamma_k+n/2-1}(s|\xi|)\;s^{ \gamma_k+n/2}ds.
\end{equation*}
We  obtain (\ref{Bes}) by applying the well known relationship between Bessel functions (Sonine's first finite integral)
   $$J_{\mu+\nu+1}(t)=\frac{t^{\nu+1}}{2^\nu\Gamma(\nu+1)}\int_0^1J_\mu(st)s^{\mu+1}(1-s^2)^{\nu}ds,$$
for $Re(\mu)>-1,\;Re(\nu)>-1$ and $t>0$
( see for example 12.11 of  \cite{Wat} ).
\end{proof}
 \par Now, to the family $\{S_z\}$ we want to apply Stein's interpolation theorem.
The following theorem is the main result of this section.
\begin{thm}\label{th1} Suppose that $0\leq \alpha\leq  \gamma_k+(n+1)/2$. Then  with a constant $C$
 $$\|S_\alpha(f)\|_{q,k}\leq C\;\|f\|_{p,k}, \qquad \text{for all}\quad f\in S(\mathbb{R}^n)$$
in each of the following cases :
\begin{itemize}
\item[(a)] For $1<p \leqslant 2 \leq  q<\infty$  and
$1/p-1/q \leq   (2\gamma_k+n+1 -2\alpha)/2(n+2\gamma_{k}) $.
 \item[(b) ]  For
 $p= (n+1+2\gamma_{k})/(n+1+2\gamma_{k}-\alpha) $ and $q=p'$
\item[(c)] For \ $1/2\leq \alpha \leq (n+1)/2+\gamma_{k}$ , on the  following cases:
\begin{itemize}
  \item[(i)] $ (n+1+2\gamma_{k})/(n+1+2\gamma_{k}-\alpha)\leq  p \leq2$ and $(n+2\gamma_{k})/q=\alpha -1/p'$.
 \item[(ii)]  $ (n+2\gamma_{k})/(n+2\gamma_k-\alpha+1/2) \leq  p \leq  (n+1+2\gamma_{k})/(n+1+2\gamma_{k}-\alpha)$   and\\
$1/q=\alpha  - (n+2\gamma_{k})/p' .$
\end{itemize}
\end{itemize}
\end{thm}
\begin{rem}
   We see that   condition  c)-(ii) can be reduced to c)-(i) by duality.
\end{rem}
  The argument proceeds similarly to the proof of Theorem 1 of  \cite{Sti}. First, we state the following  generalization
  version   of a   Hardy Littlewood Theorem.
 \begin{thm}\label{HL}
Let $0<t<2\gamma_k+n$ and  $m$ be a measurable function such that   some constant $c>0$
\begin{equation}\label{c1}
 |m(\xi)| \leq \frac{c}{|\xi|^{t}}.
\end{equation}
Then the operator  $T_{m}=\mathcal{F}_{k}^{-1}(m \mathcal{F}_{k})$ is bounded   from $L^{p}(\mathbb{R}^n, w_k(x)dx)$ to
 $L^{q}(\mathbb{R}^n, w_k(x)dx)$,  provided that
$$  1<p\leqslant2\leqslant q<\infty \ ,\
\frac{1}{p}-\frac{1}{q}=\frac{t}{2\gamma_k+n}.$$
 \end{thm}
In the case $k=0$ this result is contained in Theorem 1.11  from \cite{Horm}. In the same way   we obtain Theorem \ref{HL}    directly from the following lemmas.
 \begin{lem}\label{lem1}
Let $\varphi\geq 0$   be a measurable function such that   for   some constant $c>0$
 \begin{equation}\label{10}
   \int_{   \varphi(x) \geq s } w_k(x)dx\leqslant\frac{c}{s },\quad \forall\;s>0.
    \end{equation}
 Then for all $1 < p \leq 2$ there exists a constant $C_p>0$ such that
\begin{equation}\label{12}
  \int _{\mathbb{R}^n}\left|\frac{\mathcal{F}_k(f)(\xi)}{\varphi(\xi)}\right|^p\varphi(\xi)^2w_k(\xi)d\xi\leq C_p \|f\|_{p,k}; \quad f\in L^{p}(\mathbb{R}^n, w_k(x)dx).
\end{equation}
 \end{lem}
\begin{proof}
Put $d\mu_k(\xi)=\varphi(\xi)^2w_k(\xi)d\xi$ and $T$ the operator  $f\rightarrow T(f)=\mathcal{F}_k(f) /\varphi $. We have
 \begin{eqnarray*}
 \mu_k\Big\{\xi; \;|T(f)(\xi)|\geq s\Big\}&\leq& \mu_k\Big\{\xi;\;  \varphi(\xi)\leq \|f\|_{1,k} /s\Big\}
\\&=& \int_{ \varphi(\xi)\leq \,\|f\|_{1,k} /s  }\varphi(\xi)^2w_k(\xi)d\xi
\\&=& 2\int_{  0\,\leq\, t\,\leq \,\varphi(\xi)\,\leq\, \|f\|_{1,k} /s  } w_k(\xi) tdt d\xi
\\&\leq & 2\int_{ 0  }^{  \|f\|_{1,k} /s} t\left\{\int_{ \varphi(\xi) \,\geq \, t}  w_k(\xi)d\xi\right\} dt
\\&\leq &\frac{2c \|f\|_{1,k}}{s}.
\end{eqnarray*}
This means that the operator $T$  is bounded from $L^{1}(\mathbb{R}^n, w_k(x)dx)$ into weak space $L^{1,\infty}(\mathbb{R}^n, \mu_k(x)dx)$ .
On the other  hand from Plancherel Theorem
\begin{eqnarray*}
 \mu_k\{\xi;\; |T(f)(\xi)|\geq s\}&\leq&   \frac{\|f\|_{2,k}^2}{s^2}
\end{eqnarray*}
  Then Lemma \ref{lem1} follows  from  Marcinkiewicz  interpolation Theorem.
  \end{proof}
\begin{lem}\label{2}
 If $\varphi$ satisfies (\ref{10}) and $1 <p<r<p' < \infty$, then we have
$$ \left(\int_{\mathbb{R}^n}\left|\mathcal{F}_k(f)(\xi)(\varphi(\xi))^{(1/r-1/p')}\right|^rw_k(\xi)d\xi\right)^{1/r}\leq  C_p\|f\|_{p,k};\quad f\in L^p(\mathbb{R}^n,w_k(x)dx).$$
\end{lem}
\begin{proof}
Put $a=(p'-p)/(p'-r)$, and $a'$ it's conjugate,  we have $p/a+p'/a'=r$, $(1-r/p')a=2-p$ and $(r-p/a)a'=p'$. Then
Using  H\"{o}lder's inequality, (\ref{12}) and  the Hausdorff-Young inequality (\ref{Hausd}),
\begin{eqnarray*}
     &&\left(\int_{\mathbb{R}^n} |\mathcal{F}_k(f)(\xi)|^{\;r}|\varphi(\xi)|^{\;(1-r/p')} w_k(\xi)d\xi\right)^{1/r}
   \\ &&\leq \left(\int_{\mathbb{R}^n} |\mathcal{F}_k(f)(\xi)|^{p}|\varphi(\xi)|^{\;2-p} w_k(\xi)d\xi\right)^{1/ra} \left(\int_{\mathbb{R}^n} |\mathcal{F}_k(f)(\xi)|^{p'}  w_k(\xi)d\xi\right)^{1/ra'}
   \\&& \leq C_p \|f\|_{p,k},
\end{eqnarray*}
which is the desired statement.
  \end{proof}
\begin{lem}\label{l3}
Let  $m$ be a measurable function  and    $1<b<\infty$, such that
$$  \int_{  |m(x)|\geq s } w_k(x)dx \leq \frac{c}{s^{b}},\quad \forall\;s>0.$$
Then   the operator $T_{m}=\mathcal{F}_{k}^{-1}(m \mathcal{F}_{k})$ is  bounded    from $L^{p}(\mathbb{R}^n, w_k(x)dx)$ to
 $L^{q}(\mathbb{R}^n, w_k(x)dx)$,  provided that
$$ 1<p\leqslant2\leqslant q<\infty \quad \text{and}\quad \frac{1}{p}-\frac{1}{q}=\frac{1}{b}\;.$$
\end{lem}
\begin{proof}
We may   assume first that $p\leq q'$ and we let $\varphi=|m|^b$. Since $\varphi$ satisfies (\ref{10}), then  using Lemma \ref{2} with $r=q'$ and the fact that $1/p-1/q=1/q'-1/p'=1/b$, we obtain, for $f\in L^p(\mathbb{R}^n,w_k(x)dx),$
  $$\|m \mathcal{F}_{k}(f)\|_{q',k}\leq C_p\|f\|_{p,k}.$$
 Therefore Hausdorff-Young inequality implies
$$\|T_{m}(f)\|_{q,k}\leq \|m \mathcal{F}_{k}(f)\|_{q',k}\leq C_p\|f\|_{p,k}.$$
When $q'<p=(p')'$, we can  apply  the similar argument   to the adjoint operator $T^*_{m} = T_{\overline{m} }$,
since $ 1<q'\leqslant2\leqslant p'<\infty$ and  $1/q'- 1/p' =1/b$. Hence by duality it follows that
$$\|T_{m}(f)\|_{q,k}   \leq C_p\|f\|_{p,k}.$$
This concludes the proof of Lemma \ref{l3}.
\end{proof}
\begin{rem}
From Lemma  \ref{l3} we  obtain the statement of Theorem \ref{HL}, since
$$  \int_{  |m(x)|\geq s } w_k(x)dx \leq  \int_{  |x|\leq s^{-1/t} } w_k(x)dx\leq  \frac{c}{s^{ (2\gamma_k+n)/t}}. $$
\end{rem}
The second  fact we shall also require in proving the theorem \ref{th1} is the following
 \begin{lem}\label{lem5}
  Let $\phi\in   L^{1}(\mathbb{R}^n,w_k(x)dx)$  be a  radial function. If $\phi\in   L^{r,\infty}(\mathbb{R}^n,w_k(x)dx)$ for some  $1<r<\infty$ then the
Dunkl-convolution operator with  $\phi$ is of weak type (1,r).
\end{lem}
\begin{proof}
Let us recall that $\phi$ is in  $  L^{r,\infty}(\mathbb{R}^n,w_k(x)dx)$ if there exists a constant $c>0$ such that
 $$ \alpha _{\phi}(t)=\int_{|\phi(x)|>t}w_k(x)dx \leq \frac{c}{t^r}.$$
Let $\lambda>0$, we decompose $\phi=\phi_1+\phi_2$ where
  $$\phi_1=\left\{
    \begin{array}{ll}
     \phi, & \hbox{if $|\phi|>\lambda$,}\\
0& \hbox{if $|\phi|\leq \lambda$,}
    \end{array}
  \right.,\quad \text{and}\quad\;\phi_2= \phi-\phi_1.$$
So, we have
$$\alpha_{\phi_1}(t)=\left\{
    \begin{array}{ll}
     \alpha_{\phi}(t), & \hbox{if $t>\lambda$,}\\
\alpha_{\phi} (\lambda)& \hbox{if $t\leq \lambda$,}
    \end{array}
  \right. $$
  and
\begin{eqnarray*}
 \int_{\mathbb{R}^n}|\phi_1(x)|w_k(x)dx=\int_0^\infty\alpha_{\phi_1}(t)dt&=&\lambda \alpha_\phi(\lambda)+\int_{\lambda}^\infty \alpha_\phi(t)dt
  \leq  c \lambda^{1-r}.
\end{eqnarray*}
Then, using (\ref{trp}) we obtain for   $f\in L^{1}(\mathbb{R}^n,w_k(x)dx) $,
\begin{eqnarray}
 \|f*_k\phi_1\|_{1,k}\leq \| \phi_1\|_{1,k} \|f\|_{1,k}\leq c_0 \lambda^{1-r}\| f\|_{1,k} \label{f1}
\end{eqnarray}
and
\begin{eqnarray}
 \|f*_k\phi_2\|_{\infty}\leq \| \phi_2\|_{\infty} \|f\|_{1,k}\leq   \lambda \| f\|_{1,k}  \label{f2}
\end{eqnarray}
Now let $s>0$ and  $\lambda=s/(2\| f\|_{1,k})$. In view of (\ref{f2})
$$\int_{\{|f*_k\phi_2(x)|>s/2\}}w_k(x)dx=0$$
Thus by Chebyshev inequality and (\ref{f1}),
 \begin{eqnarray*}
 \int_{\{|f*_k\phi(x) |>s \}}w_k(x)dx &\leq& \int_{\{|f*_k\phi_1(x)|>s/2\}}w_k(x)dx+\int_{\{|f*_k\phi_2(x)|>s/2\}}w_k(x)dx\\
&\leq& 2\frac{\|f*_k\phi_1\|_{1,k}}{s}\leq c \left(\frac{\| f\|_{1,k}}{s}\right)^r,
\end{eqnarray*}
 which is the desired estimate.
\end{proof}
We are now in a position  to prove   Theorem \ref{HL}.\\
 \textbf{  Proof of  Theorem \ref{th1}.}
 Concerning Bessel function we have first to note   the following   facts.
\begin{equation}\label{ib}
 J_\nu(t)= \frac{(t/2)^\nu }{\sqrt{\pi}\;\Gamma(\nu+1/2)}\int_{-1}^1(1-u^2)^{\nu-1/2}e^{itu}du, \quad Re(\nu)>-1/2,\; t>0
\end{equation}
and
\begin{equation}\label{be}
| t^{-(\eta+i\zeta)}J_{\eta+i\zeta}(t)|\leq c_{\eta}e^{c|\zeta|}(1+t)^{-\eta-\frac{1}{2}}, \quad \eta,\; \zeta \in  \mathbb{R}\;\;\text{and}\;\;t>0 .
\end{equation}
This behavior of  Bessel function was  mentioned in \cite{Sti}. In view of  (\ref{be}) we have
$$||\xi|^{-(\gamma_k+n/2-\alpha)}J_{\gamma_k+n/2-\alpha}(|\xi|)| \leq  \frac{c}{|\xi|^{d}}; \qquad \xi\neq 0$$
for all $d \leq  \gamma_{k}+(n+1)/2-\alpha$. We thus obtain (a)  by  applying Theorem \ref{HL} to the operator $S_\alpha$.
 \par To prove $(b)$, We may apply
  Stein's  Interpolation  Theorem  to the analytic family of the operators $S_z$, for $0\leq Re(z)\leq \gamma_k+(n+1)/2$.
 Indeed, let $f$ and $g $ be a  simple functions and
   $$F(z)=\int_{\mathbb{R}^n}S_z(f)(x)g(x)w_k(x)dx,  \quad 0\leq Re(z)\leq \gamma_k+(n+1)/2.$$
By applying  the Cauchy-Schwarz Inequality and  Plancherel Theorem
the integral converges absolutely. Moreover $F$ can be  written as
$$F(z)= \int_{\mathbb{R}^n} |\xi|^{-(\gamma_k+n/2-z)}\mathcal{J}_{\gamma_k+n/2-z}(|\xi|) \mathcal{F}_k(f)(\xi) \mathcal{F}_k(\overline{g})(\xi)w_k(\xi)d\xi.$$
and so in view of  (\ref{ib}) and (\ref{be})  we see that  $F$ is analytic in $\{z\in \mathbb{C},\; 0<Re(z)<\gamma_k+(n+1)/2\}$
 and continuous in $\{z\in \mathbb{C},\; 0\leq Re(z)\leq\gamma_k+(n+1)/2\}$. Either because of  (\ref{be})  the condition (\ref{log}) holds.
Let us now consider the two  boundary  lines $Re(z)=0$ and  $Re(z)=\gamma_k+(n+1)/2$.  Using  (\ref{conv}) and the fact  that
\begin{equation}\label{gam}
 |\Gamma(1-iy)|^{-1}= \left(\frac{\pi\sinh y}{y}\right)^{1/2}\leq c e^{|y|/2}
\end{equation}
  we estimate $S _{ iy}(f)$ by
   $$\|S _{ iy}(f)\|_{\infty,k} \leq c\,e^{c \,|y|}\|f\|_{1,k}.$$
However, using (\ref{be})  and  Plancherel Theorem   we have the estimate
 \begin{eqnarray*}
 \|S _{\gamma_k+(n+1)/2+iy}(f)\|_{2,k}&\leq &c\,e^{c \,|y|}\|f\|_{2,k}.
 \end{eqnarray*}
   Therefore the application of Stein's interpolation theorem yields (b). To establish  estimate  (c) we proceed as follows: when $\alpha>1/2$ then  we obtain (c)
 from (a) and (b)  by applying  the Riesz-Thorin interpolation theorem   to the couples $(L^{p_1}, L^{q_1})$ and $(L^{p_2}, L^{q_2})$ in the two cases:
\begin{eqnarray*}
\left\{
  \begin{array}{ll}
     p_1=2, \quad q_1= (n+2\gamma_k)/(\alpha-1/2), \\
p_2=(n+2\gamma_k+1)/(n+2\gamma_k+1)-\alpha, \quad q_2=(n+2\gamma_k+1)/\alpha
  \end{array}
\right.
    \end{eqnarray*}
 and
\begin{eqnarray*}
\left\{
  \begin{array}{ll}
  p_1= (n+2\gamma_k)/(n+2\gamma_k+1/2-\alpha),\quad q_1 =2\\
p_2=(n+2\gamma_k+1)/(n+2\gamma_k+1)-\alpha,\quad q_2= (n+2\gamma_k+1)/\alpha.
  \end{array}
\right.
\end{eqnarray*}
 When $\alpha=1/2$, one can see   that  $\varphi_{1/2}\in L^{2,\infty}(\mathbb{R}^n,w_k(x)dx)$ and by
 Lemma \ref{lem5} the operator $S_{1/2}$ is of weak type $(1,2)$. Thus, according to the estimates of (b) we obtain   (c)  by  Marcinkiewicz interpolation theorem and  duality argument.  The proof of Theorem \ref{th1} is complete .
\subsection{ Estimates of Dunkl wave equation}
 In this section we are going   to apply  Theorem \ref{th1} to the Dunkl wave equation.
Our goal  is to prove Theorem \ref{thm2}.
\par We will need to consider the    Riesz transforms for Dunkl transform $R_j$ , $j=1...n$  which  defined on $L^2(\mathbb{R}^n,w_k(x)dx)$ by
$$\mathcal{F}_k(R_j(f))(\xi)=-i\frac{\xi_j}{|\xi|}\mathcal{F}_k(f)(\xi).$$
We have the following result.
\begin{thm}[\cite{AS}]\label{Riesz}
The Riesz transforms $R_j$, $1\leq j\leq n$ are bounded operators on $L^p(\mathbb{R}^n,w_k(x)dx)$ for $1 <p< \infty$.
\end{thm}
The second main auxiliary result which will be useful to prove our theorem is the following
\begin{lem}\label{ASS}
  Let $\psi$ be a radial smooth function on $\mathbb{R}^n$ such that $\psi(\xi)=0$ if $|\xi|\leq 1$ and $\psi(\xi)=1$ if $|\xi|\geq 2$. Then
the Dunkl multiplier defined by
 $\mathcal{A}_\psi(f)=\mathcal{F}^{-1}\left(\frac{\psi(\xi)}{|\xi|}\mathcal{F}_k(f)\right)$ is a bounded operator from
$L^p(\mathbb{R}^n,w_k(x)dx)$ to $L^\infty(\mathbb{R}^n,w_k(x)dx)$ for all $p>n+2\gamma_k$.
\end{lem}
\begin{proof}
Let $\rho$ be a  $C^\infty$ function  such that
$supp(\rho)\subset\{1/2\leq|\xi|\leq 2\}$ and
$$ \sum_{j=-\infty}^\infty\rho(2^{-j} \xi )=1,\qquad  \xi\neq 0.$$
Decompose,
\begin{eqnarray*}
 \frac{\psi(\xi)}{|\xi|}\mathcal{F}_k(f)(\xi)&=&\sum_{j=0}^\infty\rho(2^{-j} \xi )\frac{\psi(\xi)}{|\xi|}\mathcal{F}_k(f)(\xi)
\\&=& \left(\rho( \xi )+ \rho\left(  \frac{\xi}{2} \right)\right)\frac{\psi(\xi)}{|\xi|}\mathcal{F}_k(f)(\xi)+ \sum_{j=2}^\infty\frac{2^{-j} \rho(2^{-j} \xi )}{|2^{-j}\xi|}\mathcal{F}_k(f)(\xi)
\\&=& \psi_1(\xi)\mathcal{F}_k(f)(\xi)+ \sum_{j=2}^\infty 2^{-j} \psi_2(2^{-j}\xi)\mathcal{F}_k(f)(\xi),
\end{eqnarray*}
we get
\begin{eqnarray*}
 \mathcal{A}_{\psi}(f)= \mathcal{F}_k^{-1}(\psi_1)*_kf+ \sum_{j=2}^\infty 2^{ (n+2\gamma_k-1)j} \mathcal{F}_k^{-1}(\psi_2)(2^{ j}.)*_k f.
\end{eqnarray*}
Using  H\"{o}lder's inequality  and (\ref{trp}) it follows that, for $p>n+2\gamma_k$
   \begin{eqnarray*}
 \|\mathcal{A}_{\psi}(f)\|_\infty &\leq& \|f\|_{p,k}\left\{\|\mathcal{F}_k^{-1}(\psi_1)\|_{p',k}
+\sum_{j=2}^\infty 2^{ j( -1 + (n+2\gamma_k)/p)} \|\mathcal{F}_k^{-1}(\psi_2)\|_{p',k}\right\}
\\&\leq& C\;  \|f\|_{p,k},
\end{eqnarray*}
which is the desired result.
\end{proof}
\par We will also need the following lemma
\begin{lem}\label{14}
 Let $y\in  \mathbb{R}$ and $\Psi_j $ be the function given by $\Psi_j(x)=x_j \Phi_{iy}(x)$, $x\in\mathbb{R}^n$. Then  we can find a constant $ c> 0 $ that does not depend on $ y $ and  such that
$$\|\tau_z(\Psi_j)\|_{\infty,k}\leq  c\; e^{c|y|},$$
 for all $ z \in \mathbb {R}^n $.
 \end{lem}
\begin{proof} Let $\varepsilon>0$.  Define
 \begin{eqnarray*}
  h_\varepsilon(x)=\left\{
            \begin{array}{ll}
              e^{-\varepsilon/(1-|x|^2)}, & \hbox{if $|x|<1$,}\\
 0,& \hbox{ if $|x|\geq 1$,}
            \end{array}
          \right.
 \end{eqnarray*}
 It  follows  that
$h_\varepsilon\;\phi_{iy}$ and  $h_\varepsilon\; \phi_{-1+iy}$ are $C^\infty$-functions supported in the unit ball and
 $$\frac{\partial}{\partial x_j} \Big( \Phi_{-1+iy}(x) \;h_\varepsilon(x)\Big)=-\Psi_j(x)h_\varepsilon(x)-
 \frac{ \Psi_j (x)}{(1-iy)} \left(\frac{\varepsilon}{1-|x|^2}\;h_\varepsilon(x)\right).$$
Using the dominated convergence theorem we have the following
$$\|\Psi_j h_\varepsilon -\Psi_j\|_{2,k}\rightarrow 0,\qquad \text{as}\quad \varepsilon\rightarrow 0$$ and
$$ \left\|\left(\frac{\varepsilon}{1-|.|^2}\;h_\varepsilon \right)\Psi_j\right\|_{2,k}\rightarrow 0,\qquad \text{as}\quad \varepsilon\rightarrow 0$$
which from the  boundedness  of the Dunkl  translation opertaor $\tau_z$ on  $L^2(\mathbb{R}^n,w_k(x)dx)$   yield that
 \begin{equation}\label{pj}
 \left\|\tau_z\left(\frac{\partial}{\partial x_j} \Big( \Phi_{-1+iy} \;h_\varepsilon \Big)\right)+\tau_z( \Psi_j) \right\|_{2,k}\rightarrow 0 \qquad \text{as}\quad \varepsilon\rightarrow 0.
\end{equation}
However, since   $ h_\varepsilon\;\Phi_{-1+iy} $  is a $C^\infty$-radial function we have that
\begin{eqnarray*}
\tau_z\left(\frac{\partial}{\partial x_j} \Big( \Phi_{-1+iy}  \;h_\varepsilon \Big)\right)=
\tau_z\left( D_j^k \Big( \Phi_{-1+iy}  \;h_\varepsilon \Big)\right)=
D_j^k\tau_z  \Big( \Phi_{-1+iy}  \;h_\varepsilon \Big) .
\end{eqnarray*}
 We  next compute $D_j^k\tau_z  \Big( \Phi_{-1+iy}  \;h_\varepsilon \Big)$    and its limit when  $\varepsilon\rightarrow0$. Putting
$$A_z(x,\eta)=\sqrt{|x|^2+|z|^2-2\langle x,\eta \rangle}=\sqrt{|x-\eta|^2+|z|^2-|\eta|^2},$$
for $x\in \mathbb{R}^n$ and  $\eta \in conv(G.z)$, and
using the formula (\ref{trad}) and (\ref{Dxi}) we have that
\begin{eqnarray*}
  &&D_j^k\tau_z  \Big( \Phi_{-1+iy}  \;h_\varepsilon \Big) (x)\\&&=-\int_{\mathbb{R}^n}(x_j-\eta_j) \widetilde{\Phi_{iy}}(A_z(x,\eta))\widetilde{h_\varepsilon}(A_z(x,\eta))d\nu_z(\eta)\\&&\qquad \qquad \quad-\int_{\mathbb{R}^n}
\left(\frac{(x_j-\eta_j)\widetilde{\Phi_{iy}}(A_z(x,\eta)) }{(1-iy)}\right)\left(\frac{\varepsilon}{1-A_z(x,\eta)^2}\;\widetilde{h_\varepsilon}(A_z(x,\eta))\right)d\nu_z(\eta)\\&&
+ \sum_{\upsilon\in R^+}\frac{k_\upsilon\upsilon_j}{\langle x, \upsilon \rangle} \int_{\mathbb{R}^n}\Big(\widetilde{\Phi_{-1+iy}}(A_z(x,\eta))\widetilde{h_\varepsilon}(A_z(x,\eta))\\&&\qquad \qquad
\qquad\qquad\qquad\qquad\qquad \quad -
\widetilde{\Phi_{-1+iy}}(A_z(\sigma_\upsilon.x,\eta))\widetilde{h_\varepsilon}(A_z(\sigma_\upsilon.x,\eta))\Big)
d\nu_z(\eta).
\end{eqnarray*}
Therefore, from (\ref{pj})and   dominated convergence Theorem we obtain  for a.e. $x\in \mathbb{R}^n$,
\begin{eqnarray} \label{118}
  &&\tau_z( \Psi_j)(x)= \int_{\mathbb{R}^n}(x_j-\eta_j)\widetilde{ \Phi_{iy}}(A_z(x,\eta)) d\mu_z(\eta) \nonumber \\&&
 - \sum_{\upsilon\in R^+}\frac{k_\upsilon\upsilon_j}{\langle x, \upsilon \rangle}  \int_{\mathbb{R}^n}\Big(\widetilde{\Phi_{-1+iy}}(A_z(x,\eta)) -
\widetilde{\Phi_{-1+iy}}(A_z(\sigma_\upsilon.x,\eta)) \Big)
d\nu_z(\eta).
\end{eqnarray}
Note   that   in the integrands, $|(x_j-\eta_j)|\leq  A_z(x,\eta)  \leq 1 $ and by using  (\ref{gam})
$$|(x_j-\eta_j)\widetilde{ \Phi_{iy}}(A_z(x,\eta))|\leq c\;e^{c|y|}.$$
Also, if we write
\begin{eqnarray*}
  &&  \frac{\widetilde{\Phi_{-1+iy}}(A_z(x,\eta)) -
\widetilde{ \Phi_{-1+iy}}(A_z(\sigma_\upsilon.x,\eta))}{\langle x, \upsilon \rangle}
\\&&\qquad\qquad\qquad\qquad\qquad =-\sum_{j=1}^n\int_0^1(x_j-t\langle x,\upsilon\rangle\upsilon_j-\eta_j)\upsilon_j\widetilde{ \Phi_{ iy}}(A_z(x-t\langle x,\upsilon\rangle\upsilon,\eta))dt
\end{eqnarray*}
then   we have that
\begin{eqnarray*}
 \left| \frac{\widetilde{\Phi_{-1+iy}}(A_z(x,\eta)) -
\widetilde{\Phi_{-1+iy}}(A_z(\sigma_\upsilon.x,\eta))}{\langle x, \upsilon \rangle}\right|\leq c\;e^{c|y|}.
\end{eqnarray*}
Thus in view of  (\ref{118}), we  conclude  the proof of Lemma \ref{14}.
\end{proof}
 \textbf{Proof of Theorem \ref{thm2}. }
Applying Theorem \ref{th1}- (c), yields
$$ \|S_{ \gamma_k+(n-1)/2}(f)\|_{q,k}\leq c \|f\|_{p,k};\qquad f\in S(\mathbb{R}^n)$$
 for all couple
$(p,q)$ satisfying (\ref{q1}) or (\ref{q2}). So in view of (\ref{sol}) it will be enough to prove
\begin{equation}\label{eq3}
 \|S_{ \gamma_k+(n+1)/2}(f)\|_{q,k}\leq c  \left\| \sum_{j=1}^n|D_j^kf|\right\|_{p,k};\qquad f\in S(\mathbb{R}^n).
\end{equation}
We quote the following
$$\cos( \xi)\mathcal{F}_k(f)(\xi)= \sum_{j=1}^n \frac{\xi_j^2}{|\xi|^2}\cos( \xi)\mathcal{F}_k(f)(\xi)
= \sum_{j=1}^n \frac{\cos( \xi)}{|\xi|}\mathcal{F}_k\Big(R_j(D_j^kf)\Big)(\xi).$$
Hence from Theorem \ref{Riesz}  one can   reduce (\ref{eq3})    to show that
\begin{equation}\label{eq4}
 \left\|\mathcal{F}_k^{-1}\left( \frac{\cos( \xi)}{|\xi|}\mathcal{F}_k(f)(\xi)\right)\right\|_{q,k}\leq c \|  f\|_{p,k}.
\end{equation}
 Let $\psi$ be a radial smooth function on $\mathbb{R}^n$ such that $\psi(\xi)=0$ if $|\xi|\leq 1$ and $\psi(\xi)=1$ if $|\xi|\geq 2$.  Then the theorem \ref{HL}
implies
\begin{equation}\label{120}
 \left\|\mathcal{F}_k^{-1}\left( (1-\psi(\xi))\frac{\cos( \xi)}{|\xi|}\mathcal{F}_k(f)(\xi)\right)\right\|_{p,k}\leq c \|  f\|_{p,k},
\end{equation}
provide, $1<p\leq 2\leq q<\infty $ and  $1/p-1/q\geq 1/(2\gamma_k+n)$. Here clearly  conditions (\ref{q1}) and (\ref{q2}) are also  satisfied.
Thus we are reduced to showing that
\begin{equation*}
 \left\|\mathcal{F}_k^{-1}\left(  \psi(\xi)\frac{\cos( \xi)}{|\xi|}\mathcal{F}_k(f)(\xi)\right)\right\|_{p,k}\leq c \|  f\|_{p,k}.
\end{equation*}
For this purpose we  define an analytic family of linear operators  $U_z^j$  by
$$U_z^j(f)=\mathcal{F}_k^{-1}\Big(\psi(\xi) \xi_j|\xi|^{- n/2-\gamma_k+z-1}J_{ n/2+\gamma_k-z-1}(|\xi|)\mathcal{F}_k(f)(\xi)\Big)$$
for $0\leq Re(z)\leq \gamma_k+(n+1)/2$. To this family we  apply Stein's interpolation theorem, and proceeding as in the proof of
Theorem \ref{th1}. First on the boundary   $Re(z)=\gamma_k+(n+1)/2$ we have
$$\|U_{z}^j(f)\|_{2,k}\leq c\; e^{c|y|}\;\|f\|_{2,k} $$
which   is a simple consequence of (\ref{be}) and Plancherel Theorem.
\par For  $z=iy$,  in view of (\ref{be}), the function $\xi \rightarrow \psi(\xi) \xi_j|\xi|^{- n/2-\gamma_k+iy-1}J_{ n/2+\gamma_k-iy-1}(|\xi|)$ belongs to  $L^2(\mathbb{R}^n,w_k(x)dx)$. Then  one can write $U_{iy}^j$ as the convolution operator
$$U_{iy}^j(f)(x)=\mathcal{F}_k^{-1}\Big(\psi(\xi) \xi_j|\xi|^{- n/2-\gamma_k+z-1}J_{ n/2+\gamma_k-z-1}(|\xi|\Big)*_kf(x),\quad f\in L^2(\mathbb{R}^n,w_k(x)dx)$$
and to obtain a desired  $L^1-L^\infty$ estimate  for $U_{iy}^j$ as in Theorem \ref{sstt}
 it suffices to estimate
$$\left\|\tau_x \mathcal{F}_k^{-1}\Big(\psi(\xi) \xi_j|\xi|^{- n/2-\gamma_k+z-1}J_{ n/2+\gamma_k-z-1}(|\xi|\Big) \right\|_{\infty,k}.$$
We begin by recalling   the two classical identities for  Bessel function
 \begin{equation}\label{111}
\frac{d}{dt}(t^{-\nu} J_\nu(t))= -t^{-\nu} J_{\nu+1}(t)
\end{equation}
and
 \begin{equation}\label{112}
 J_{\nu+1}(t)=  2\nu J_\nu(t)/t-J_{\nu-1}(t).
\end{equation}
Let us observe first that  from  (\ref{Bes}) and  (\ref{111})
\begin{eqnarray}\label{113}
 \mathcal{F}_k(x_j\Phi_{iy})(\xi)=i D_j^k\mathcal{F}_k( \Phi_{iy})(\xi)=-i \xi_j |\xi|^{ -\gamma_k-n/2+iy-1 }J_{ \gamma_k+n/2-iy +1 }(|\xi|),
\end{eqnarray}
 and from (\ref{112})
\begin{eqnarray*}
 \mathcal{F}_k(x_j\Phi_{iy})(\xi)&=&i \xi_j |\xi|^{ -\gamma_k-n/2+iy-1 }J_{ \gamma_k+n/2-iy -1 }(|\xi|)
\\&&\qquad -i  (n+2\gamma_k-iy)\xi_j |\xi|^{ -\gamma_k-n/2+iy-2 }J_{ \gamma_k+n/2-iy   }(|\xi|).
\end{eqnarray*}
Hence
\begin{eqnarray*}
&&\mathcal{F}_k^{-1}\Big(\psi(\xi)\xi_j |\xi|^{ -\gamma_k-n/2+iy-1 }J_{ \gamma_k+n/2-iy -1 }(|\xi|)\Big)
 =-i \mathcal{F}_k^{-1}\Big(\psi(\xi) \mathcal{F}_k(x_j\Phi_{iy})(\xi)\Big)
\\&& +(n+2\gamma_k-iy)\mathcal{F}_k^{-1}\Big(\psi(\xi)\xi_j |\xi|^{ -\gamma_k-n/2+iy-2 }J_{ \gamma_k+n/2-iy   }(|\xi|)\Big).
\end{eqnarray*}
 Now,  write
\begin{eqnarray*}
 \mathcal{F}_k^{-1}\Big(\psi(\xi) \mathcal{F}_k(x_j\Phi_{iy} )(\xi)\Big)
=\mathcal{F}_k^{-1}\Big((\psi(\xi)-1) \mathcal{F}_k(x_j\Phi_{iy} )(\xi)\Big)
+ x_j\Phi_{iy} (x).
 \end{eqnarray*}
The function $\xi\rightarrow (\psi(\xi)-1) \mathcal{F}_k(x_j\phi_{iy})(\xi)$ is a $C^\infty$ with compact support,
so by  using (\ref{trs}),   (\ref{be}) and  (\ref{113}) it follows that
$$\left\|\tau_x\Big\{\mathcal{F}_k^{-1}(  (\psi(\xi)-1) \mathcal{F}_k(x_j\Phi_{iy})(\xi))\Big\}\right\|_{\infty,k}\leq c\;e^{c|y|}.$$
Thus, in view   Lemma \ref{14}
\begin{equation}\label{114}
   \left\|\tau_x\mathcal{F}_k^{-1}\Big(\psi(\xi) \mathcal{F}_k(x_j\Phi_{iy})(\xi)\Big)\right\|_{\infty, k}\leq c\;e^{c|y|}.
\end{equation}
On the other hand, one can write
\begin{eqnarray*}
\mathcal{F}_k^{-1}\Big(\psi(\xi)\xi_j |\xi|^{ -\gamma_k-n/2+iy-2 }J_{ \gamma_k+n/2-iy   }(|\xi|)\Big)&=&
\mathcal{F}_k^{-1} \left(\frac{\xi_j}{|\xi|}\left(  \frac{\psi(\xi)}{|\xi|} \mathcal{F}_k(\Phi_{iy})(\xi)\right)\right) \\
&=&i\mathcal{A}_\psi(R_j(\Phi_{iy}))
\end{eqnarray*}
Now for  $p>n+2\gamma_k$ the  radial function $\phi_{iy}$  belongs to $L^p(\mathbb{R}^n,w_k(x)dx)$, thus we can apply Theorem \ref{Riesz},
Lemma \ref{ASS} and (\ref{Y}) to obtain
\begin{eqnarray*}
\left\|\tau_x\mathcal{F}_k^{-1}\Big(\psi(\xi)\xi_j |\xi|^{ -\gamma_k-n/2+iy-2 }J_{ \gamma_k+n/2-iy   }(|\xi|)\Big)\right\|_{\infty,k}
&=&\|\mathcal{A}_\psi(R_j(\tau_x(\phi_{iy})))\|_{\infty,k}\\&\leq& c \|\phi_{iy})\|_{p,k}\leq c\; e^{c|y|}.
\end{eqnarray*}
This together  with (\ref{114}) yield
$$\left\|\tau_x \mathcal{F}_k^{-1}\Big(\psi(\xi) \xi_j|\xi|^{- n/2-\gamma_k+z-1}J_{ n/2+\gamma_k-z-1}(|\xi|\Big) \right\|_{\infty,k}\leq c\;e^{c|y|}$$
and therefore
$$\|U_{iy}^j(f)\|_{\infty,k}\leq c\; e^{c|y|}\;\|f\|_{1,k}.$$
The Stein  interpolation theorem now implies the following
$$\|U_{\alpha}^j(f)\|_{p',k}\leq c\;  \|f\|_{p,k}$$
for all $0\leq \alpha\leq\gamma_k+(n+1)/2$ and
 $p= (n+1+2\gamma_{k})/(n+1+2\gamma_{k}-\alpha) $.
In particular for $\alpha=\gamma_k+(n-1)/2$ and $p=2(n+2\gamma_k+1)/(n+2\gamma_k+3)$we have that
\begin{equation}\label{115}
  \left\|\mathcal{F}_k^{-1}\left(\psi(\xi)\xi_j \frac{ \cos(|\xi|)}{|\xi|^{  2}}\mathcal{F}_k(f)(\xi)\right)\right\|_{p',k}\leq c\;  \|f\|_{p,k}
\end{equation}
Using the fact that
 \begin{eqnarray*}
  \mathcal{F}_k^{-1}\left(\psi(\xi)  \frac{ \cos(|\xi|)}{|\xi| }\mathcal{F}_k(f)(\xi)\right)
&=&\sum_{j=1}^n \mathcal{F}_k^{-1}\left(\psi(\xi)\frac{\xi^2_j}{|\xi|^2}  \frac{ \cos(|\xi|)}{|\xi| }\mathcal{F}_k(f)(\xi)\right)
\\&=&\sum_{j=1}^n \mathcal{F}_k^{-1}\left(\psi(\xi) \xi_j \frac{ \cos(|\xi|)}{|\xi|^2 }\mathcal{F}_k(R_j(f))(\xi)\right)
 \end{eqnarray*}
it follows from (\ref{115}) and Theorem  \ref{Riesz}
\begin{equation}\label{116}
  \left\|\mathcal{F}_k^{-1}\left(\psi(\xi)  \frac{ \cos(|\xi|)}{|\xi| }\mathcal{F}_k(f)(\xi)\right)\right\|_{p',k}\leq c\;  \|f\|_{p,k}
\end{equation}
On the other  hand from the  theorem \ref{HL} we have
\begin{equation}\label{117}
  \left\|\mathcal{F}_k^{-1}\left(\psi(\xi)  \frac{ \cos(|\xi|)}{|\xi| }\mathcal{F}_k(f)(\xi)\right)\right\|_{q,k}\leq c\;  \|f\|_{p,k}
\end{equation}
for $1<p\leq 2\leq q< \infty$, $1/p-1/q=1/(n+2\gamma_k)$. Therefore  with   the use of
 the Riesz-Thorin interpolation theorem   for  the couples $(L^{p_1}, L^{q_1})$ and $(L^{p_2}, L^{q_2})$ for
\begin{eqnarray*}
\left\{
  \begin{array}{ll}
     p_1=2, \quad q_1= 2(n+2\gamma_k)/(n+2\gamma_k-2), \\
p_2=2(n+2\gamma_k+1)/(n+2\gamma_k+3) , \quad q_2=2(n+2\gamma_k+1)/(n+2\gamma_k-1)
  \end{array}
\right.
    \end{eqnarray*}
 and
\begin{eqnarray*}
\left\{
  \begin{array}{ll}
  p_1= 2(n+2\gamma_k)/(n+2\gamma_k++2),\quad q_1 =2\\
p_2=2(n+2\gamma_k+1)/(n+2\gamma_k+3) , \quad q_2=2(n+2\gamma_k+1)/(n+2\gamma_k-1)  \end{array}
\right.
\end{eqnarray*}
we obtain
\begin{equation}\label{117}
  \left\|\mathcal{F}_k^{-1}\left(\psi(\xi)  \frac{ \cos(|\xi|)}{|\xi| }\mathcal{F}_k(f)(\xi)\right)\right\|_{q,k}\leq c\;  \|f\|_{p,k}
\end{equation}
for all $p$ and $q$ satisfying (\ref{q1}) or (\ref{q2}).   This combined with the estimate   (\ref{120}) yields (\ref{eq4}) and finishes  the proof of Theorem \ref{thm2}.

 \end{document}